\documentclass[preprint,12pt]{elsarticle}
\usepackage{color}
\usepackage{amsthm,amsbsy,amsmath,amsfonts,amssymb,amscd}
\RequirePackage{graphicx}
\RequirePackage{caption}
\RequirePackage{subcaption}
\usepackage{hyperref}
\newtheorem{proposition}{Proposition}
\newtheorem{definition}{Definition}
\newtheorem{remark}{Remark}
\newdefinition{rmk}{Remark}
\newtheorem{example}{Example}
\newtheorem{corollary}{Corollary}
\begin{document}
\begin{frontmatter}
\title{Invariances of random fields paths, with applications in 
Gaussian Process Regression}

\author[imsv]{David Ginsbourger\corref{cor1}}
\ead{ginsbourger@stat.unibe.ch}

\author[emse]{Olivier Roustant}
\ead{roustant@emse.fr}

\author[sitran]{Nicolas Durrande}
\ead{n.durrande@sheffield.ac.uk}

\address[imsv]{University of Bern, Department of Mathematics and Statistics, 
Alpeneggstrasse, 22 CH-3012 Bern, Switzerland}
\address[emse]{Ecole Nationale Sup\'erieure des Mines de Saint-Etienne, 
FAYOL-EMSE, LSTI, F-42023 Saint-Etienne, France}
\address[sitran]{The University of Sheffield,  Department of Computer Science, 
Regent Court, 211 Portobello, Sheffield S1 4DP, UK}

\cortext[cor1]{Corresponding author}

\begin{abstract}
We study pathwise invariances of centred random fields that can be controlled through the covariance.  
A result involving composition operators is obtained in second-order settings, and we show that various path properties including additivity boil down to invariances of the covariance kernel. 
These results are extended to a broader class of operators in the Gaussian case, via the Lo\`eve isometry. 
Several covariance-driven pathwise invariances are illustrated, including fields with symmetric paths, centred paths, harmonic paths, or sparse paths.  
The proposed approach delivers a number of promising results and perspectives in Gaussian process regression. 
\end{abstract}


\end{frontmatter}

\section{Introduction}

Whether for function approximation,  classification, or density estimation, probabilistic models relying on random fields have been increasingly used in recent works from various research communities. Finding their roots in geostatistics and spatial statistics with optimal linear prediction and Kriging \citep{mat63, ste99}, random field models for prediction have become a main stream topic in machine learning (under the \textit{Gaussian Process Regression} terminology, see, e.g., \cite{ras:wil06}), with a spectrum ranging from metamodeling and adaptive design approaches for time-consuming simulations in science and engineering \cite{wel:buc:sac:wyn:mit:mor92, oha06, jon01, san:wil:not03}) to theoretical Bayesian statistics in function spaces (See \cite{van:van2008a, van:van2008b, van:van2011} and references therein). 
Often, a Gaussian random field model is assumed for the function of interest, and so all prior assumptions on this function are incorporated through the corresponding mean function and covariance kernel. 
Here we focus on random field models for which the covariance kernel exists, and we discuss some mathematical properties of associated realisations (or \textit{paths}) depending on the kernel, both in the Gaussian case and in a more general second-order framework.

\medskip

A number of well-known random field properties driven by the covariance kernel (say in the centred case) are in the mean square sense \cite{ssd}, e.g. $L^2$ continuity and differentiability \cite{review_corr}. 
Such results are quite general in the sense that they hold in a variety of cases (Gaussian or not), but they generally aren't informative about the pathwise beahviour of underlying random fields. 
In the Gaussian case, however, much can be said about path regularity properties of stationary random field paths (Cf. classical results in \cite{cra:lea1967} and subsequent works) based on  
the behaviour of the covariance kernel in the neighbourhood of the origin. 
Likewise, for non-stationary Gaussian fields, results connecting path regularity to kernel properties can be found in \cite{adl1990}. 
More recently, \cite{sch2010} studied path regularity of second-order random fields, and could draw conclusions about a.s. continuous differentiability  in non-Gaussian settings. Also, we refer to the thesis \cite{sch2009} for an enlightening exposition of state-of-the-art results concerning regularity properties of random field sample paths in various frameworks. 

In a different settings, links between invariances of kernels under operations like translations and rotations (that is to say, the notions of \textit{stationarity} and \textit{isotropy} \cite{ssd}) and invariances in distribution of the corresponding random fields have been covered extensively in spatial statistics and throughout the literature of probability theory \cite{partha}. However, such properties are to be understood \textit{in distribution}, and do not directly concern random field paths. 
Our main focus in the present work is on \textit{pathwise} algebraic and geometric properties of random fields, such as invariances under group actions or sparse function decompositions of multivariate paths. 

\medskip

We first establish in a quite general framework, that for a centred random field $(Z_{x})_{x\in D}$ possessing a covariance (i.e. such that the variance is finite at any location in the index space $D$), $Z$ has paths invariant under $T$ with probability $1$ if and only if $\forall \mathbf{x}\in D \  T(k(\cdot, \mathbf{x}))=k(\cdot, \mathbf{x})$, where $T$ belongs to the class of linear combinations of composition operators. 
The presented results generalise \cite{afst_david}, where random fields with paths invariant under the action of a finite group were studied. 
Here we also extend recent works on additive kernels for high-dimensional kriging \cite{afst_nico,Duvenaud2011additif}, and we provide a simple characterization of the class of kernels leading to squared-integrable centred random fields with additive paths.
Furthermore, in the particular case of Gaussian random fields, a more general class of invariances can be covered through the link between operators on the paths and operators on the \textit{reproducing kernel Hilbert space} \cite{ber:tho04} 
associated with the random field. 

\medskip

Section $2$ presents a general result characterizing path invariance in terms of argumentwise invariance of covariance kernels, in the case of combinations of composition operators. 
In Section $3$, we discuss how the Gaussianity assumption enables extending the results of Section $2$ to more general operators.
The obtained results are applied to Gaussian process regression in Section $4$ where the potential of argumentwise invariant kernels is demonstrated through various examples. 
Section $5$ is dedicated to an overall conclusion of the article, and a discussion on some research perspectives. 

\section{Invariance under combinations of composition operators}
\label{sec:coc}
\subsection{Motivations}

Designing kernels imposing some structural constraints on the associated random field models is of interest in various situations. 

One of those situations is the high-dimensional function approximation framework, where simplifying assumptions are needed in order to guarantee 
a reasonable inference despite the curse of dimensionality. 
Following its successful use in multidimensional nonparametric smoothing \cite{stone1985additive, Hastie1990}, the \textit{additivity assumption} has become a very popular simplifying assumption for dealing with high-dimensional problems, and has recently inspiring further work in mathematical statistics \cite{mei:van:bue2009, ras:wai:yu2011, gay:ing2012}. 

A class of kernels leading to random fields with additive paths, in the sense detailed below (See also \cite{liu2013}), has recently been considered in \cite{afst_nico}. 
Calling a function $f \in \mathbb{R}^D$ (with multidimensional source space $D=\prod_{i}^{d} D_{i}$ where $D_{i} \subset \mathbb{R}$) additive when there exists $f_{i} \in \mathbb{R}^{D_{i}} \ (1\leq i \leq d)$ such that 
$\forall \mathbf{x}=(x_{1}, \ldots, x_{d}) \in D, \ {f(\mathbf{x}) = \sum_{i=1}^{d} f_{i}(x_{i})}$, it was indeed shown in \cite{afst_nico} that 
\begin{proposition}
\label{propadd}
If a  random field  $Z$ possesses a kernel of the form
\begin{equation}
k(\mathbf{x},\mathbf{x}') = \sum_{i=1}^d k_i(x_i,x_i')
\label{eq1_add_kernels}
\end{equation} 
where the $k_i$'s are arbitrary positive definite kernels over the $D_{i}$'s, 
then $Z$ is additive up to a modification, i.e. there exists a random field $A$ which paths 
are additive functions such that  
$\forall \mathbf{x}\in D \ \mathbb{P}(Z_\mathbf{x} =A_\mathbf{x}) = 1$.
\label{prop1}
\end{proposition}

\noindent
One may wonder whether kernels of the form $k(\mathbf{x},\mathbf{x}') = \sum_{i=1}^d k_i(x_i,x_i')$ 
are the only ones giving birth to centred random fields with additive paths. 
The answer to this question turns out to be negative,  as will be established in Corollary \ref{additive_dependent}.   

\medskip

\noindent
In a similar fashion, \cite{afst_david} gives a characterization of kernels which associated centred random fields have their paths invariant under a finite group action. 
\noindent
Let $G$ be a finite group acting on D via a measurable action 
\begin{equation*}
\Phi : (\mathbf{x},g) \in D\times G \longrightarrow \Phi(\mathbf{x},g)=g.\mathbf{x} \in D
\end{equation*}

\begin{proposition}
$Z$ has invariant paths under $\Phi$ (\textit{up to a modification}) if and only if $k$ is argumentwise invariant: 
$\forall (\mathbf{x}, g) \in D \times G, \ k(g.\mathbf{x}, \cdot) = k(\mathbf{x}, \cdot)$.
\label{prop2}
\end{proposition}

\noindent
We show in Proposition \ref{prop:kernelProcess} that both Propositions \ref{prop1} and \ref{prop2} are sub-cases of a general result on square-integrable random fields invariant under the class of \textit{combination of composition operators} (see Definition~\ref{def:comofcom}). 
A characterization of kernels leading to random fields possessing additive paths is given in Corollary \ref{additive_dependent}, and it is then shown that having the form of Eq. \ref{eq1_add_kernels} is not necessary. Another by-product of Proposition \ref{prop:kernelProcess} is a new proof of Proposition~\ref{prop2} relying on a particular class of combination of composition operators, as illustrated in Example~\ref{ex_CCO_GroupInvariance}. Let us now introduce the set up of composition operators (See, e.g., \cite{sin:man93})  and their combinations. 

\subsection{Composition operators and their combinations}

\begin{definition}
Let us consider an arbitrary function $v: \mathbf{x} \in D \longrightarrow v(\mathbf{x}) \in D$.
The \textit{composition operator $T_{v}$ with symbol $v$} is defined as follows: 
\begin{equation*}
T_{v}: f \in \mathbb{R}^{D} \longrightarrow T_{v}(f) = f \circ v \in \mathbb{R}^{D}
\end{equation*}
\end{definition}

\begin{remark}
\noindent
Such operators can be naturally extended to random fields indexed by $D$: 
\begin{equation*}
\forall \mathbf{x} \in D, \ T_{v}(Z)_{\mathbf{x}} = Z_{v(\mathbf{x})}
\end{equation*}
\end{remark}

\begin{definition}
We call \textit{combination of composition operators with symbols $v_{i} \in D^D$ 
and weights $\alpha_{i} \in \mathbb{R} \ 
(1\leq i \leq q, q \in \mathbb{N}\backslash\{0\})$} the operator 
\begin{equation*}
T = \sum_{i=1}^{q} \alpha_{i} T_{v_{i}}
\end{equation*}
\label{def:comofcom}
\end{definition}
%
\subsection{Invariance under a combination of composition operators}

\begin{proposition}
\label{prop:kernelProcess}
Let $Z$ be a square-integrable centred random field with covariance kernel $k$. Then $Z$ equals $T(Z)$ up to a modification, i.e. 
\begin{equation*}
\forall \mathbf{x} \in D, \ \mathbb{P}\left( Z_{\mathbf{x}}= T(Z)_{\mathbf{x}} \right) = 1
\end{equation*}
\noindent
if and only if $k$ is $T$-invariant, i.e. 
\begin{equation}
\forall \mathbf{x}' \in D, \ T(k(\cdot,\mathbf{x}')) = k(\cdot,\mathbf{x}').
\end{equation}
\end{proposition}

\begin{proof}
\noindent
$\Rightarrow$: let us fix arbitrary $\mathbf{x},\ \mathbf{x}' \in D$. Since $Z_{\mathbf{x}}$ is a modification of $T(Z)_{\mathbf{x}}$, we have $\text{cov}(Z_{\mathbf{x}}, Z_{\mathbf{x}'})=\text{cov}(T(Z)_{\mathbf{x}}, Z_{\mathbf{x}'}) = \text{cov}(\sum_{i=1}^{q} \alpha_{i}Z_{v_{i}(\mathbf{x})}, Z_{\mathbf{x}'} )$, and so:
$$ k(\mathbf{x},\mathbf{x}') = \sum_{i=1}^{q} \alpha_{i} k( v_{i}(\mathbf{x}), \mathbf{x}') = T(k(\cdot,\mathbf{x}')) (\mathbf{x}).$$

\noindent
$\Leftarrow$: Using $\forall \mathbf{x}' \in D \ T(k(\cdot,\mathbf{x}')) = k(\cdot,\mathbf{x}')$, we get 
$\text{var}(T(Z)_{\mathbf{x}}) = \text{cov}(Z_{\mathbf{x}}, T(Z)_{\mathbf{x}}) = \text{var}(Z_{\mathbf{x}})$, so 
$\text{var}(Z_\mathbf{x} - T(Z)_\mathbf{x}) = 0$. 
Since $Z$ is centred, so is $T(Z)$, and hence $Z_{\mathbf{x}}  \stackrel{a.s.}{=}  T(Z)_{\mathbf{x}}$.
\end{proof}

\begin{remark}
As noted in \cite{rev:yor91}, two processes modifications of each other that are almost surely continuous are indistinguishable. Almost sure results may then directly be obtained for processes with almost surely continuous paths.  
\end{remark}

\begin{example}[Case of group-invariance]
Prop.~\ref{prop2} now appears as a special case of 
Prop.~\ref{prop:kernelProcess} with $T(f) = \sum_{i=1}^{\# G} \frac{1}{\# G} f(v_{i}(\cdot))$ where $v_{i} (\mathbf{x}) := g_{i}.\mathbf{x} \ (1\leq i \leq \# G) $.

For instance, let us consider the following functions over $[-1,1]^2 \times [-1,1]^2$:
\begin{equation}
\begin{split}
k_1(\mathbf{x},\mathbf{y}) & = min(\rho_\mathbf{x} \cos(\tilde{\theta}_\mathbf{x}), \rho_\mathbf{y} \cos(\tilde{\theta}_\mathbf{y})) \times min(\rho_\mathbf{x} \sin(\tilde{\theta}_\mathbf{x}), \rho_\mathbf{y} \sin(\tilde{\theta}_\mathbf{y})) \\
k_2(\mathbf{x},\mathbf{y}) & = min(\rho_\mathbf{x} , \rho_\mathbf{y} )
\end{split}
\end{equation}
where $\rho_{\mathbf{x}},\theta_{\mathbf{x}}$ (resp. $\rho_\mathbf{y},\theta_\mathbf{y}$) are the polar coordinates of $\mathbf{x}$ (resp. $\mathbf{y}$) and where $\tilde{\theta} = \theta \ \mathrm{mod} \ \pi/2$. 
$k_{1}$ and $k_{2}$ are positive definite kernels (in the loose sense) as admissible covariances (respectively those of the Brownian Sheet and the Brownian Motion) composed with a change of index, i.e. $k_i(\cdot,\cdot) = k(h(\cdot),h(\cdot))$ with $h$ from $[-1,1]^2$ onto $[0,1]^2$ (resp. from $\mathbb{R}^2$ onto $\mathbb{R}_+$). 
By construction $k_1$ and $k_2$ are argumentwise invariant with respect to rotations around the origin (with angles multiple of $\pi/2$ for $k_1$). As illustrated in Figure~\ref{fig:ex_brown_rot}, Proposition~\ref{prop:kernelProcess} ensures that the sample paths of centred (Gaussian or non-Gaussian) random fields with these kernels inherit their invariance properties. 
Note that $k_{1}$ belongs to the class of kernels argumentwise invariant under the action of a finite group treated above, while the argumentwise invariance of $k_{2}$ under the action of an infinite group can actually be seen as invariance under any composition operator with symbol of the form 
$v(\mathbf{x})=\rho_{\mathbf{x}} \mathbf{u}$ where $\mathbf{u}$ is an arbitrary point on the unit circle. 

\begin{figure}
\centering
\begin{subfigure}[b]{0.45\textwidth}
    \centering
    \includegraphics[width=\textwidth]{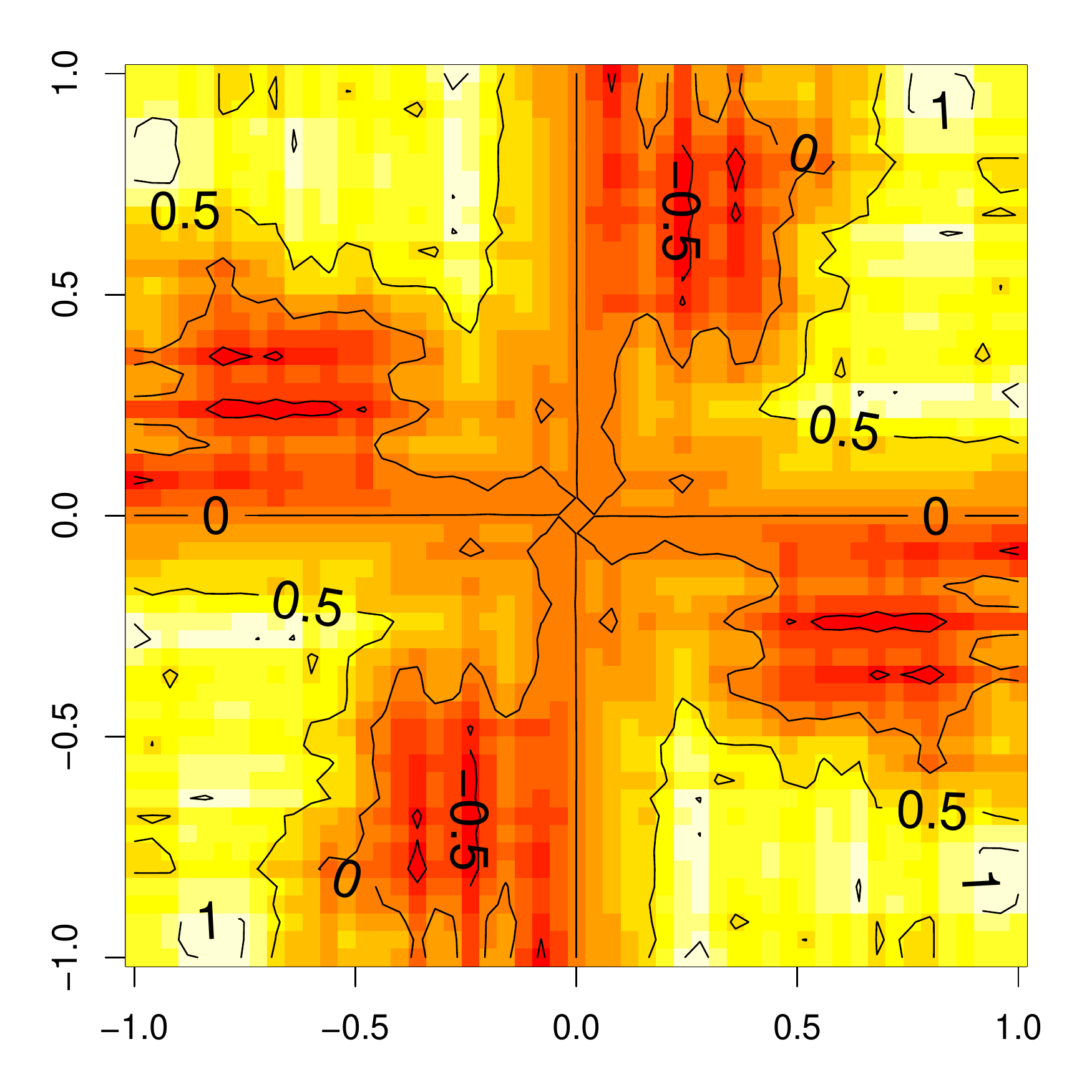}
    \caption{Path of a GRF with kernel $k_1$}
\end{subfigure} \qquad
\begin{subfigure}[b]{0.45\textwidth}
    \centering
    \includegraphics[width=\textwidth]{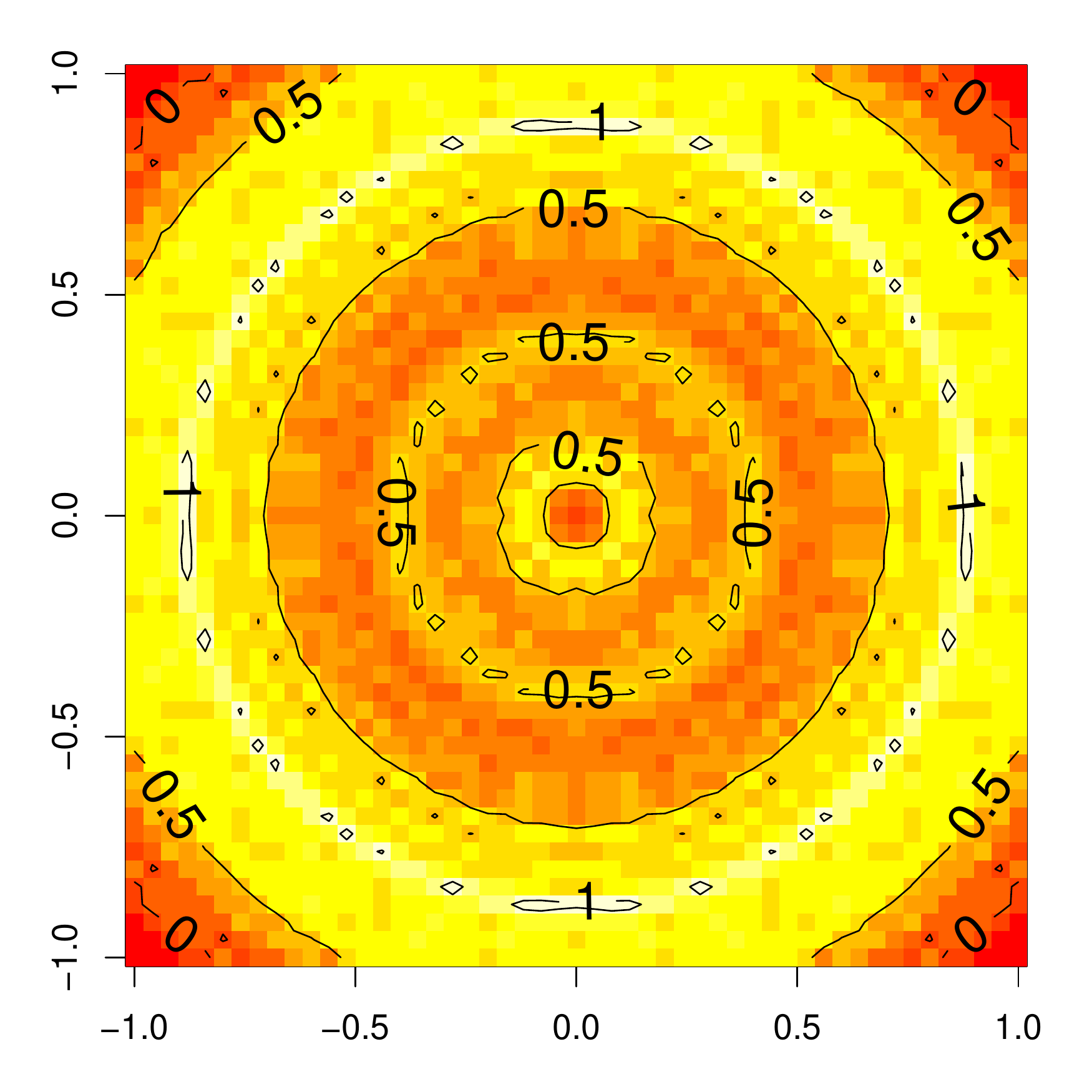}
    \caption{Path of a GRF with kernel $k_2$}
\end{subfigure}%
\caption{Sample paths of centred (Gaussian) random fields with two kinds of argumentwise rotation invariant kernels. The fact that the contour lines are not exactly rotation invariant on the right panel can be explained by the 2-dimensional mesh used for plotting.}
\label{fig:ex_brown_rot}
\end{figure}

\label{ex_CCO_GroupInvariance}
\end{example}

Beyond group-invariance, Proposition~\ref{prop:kernelProcess} also has implications concerning the sparsity of multivariate random field paths, as detailed in the following section on additivity, leading to a generalization of Proposition~\ref{propadd}.

\subsection{Kernels of centred random fields with additive paths}

Let us first show how the additivity property boils down to an invariance property under some specific class of combination of composition operators. 

\begin{remark}
\label{rmk:additivefun}
Assuming $\mathbf{a} \in D$, a function $f: D \rightarrow \mathbb{R}$ is additive if and only if  
$f$ is invariant under the following combination of composition operators:
\begin{equation}
 T(f)(\mathbf{x}) = \sum_{i=1}^{d} f( v_{i} (\mathbf{x}) ) - (d-1) f(v_{d+1}(\mathbf{x})) \ \ \ \ (\mathbf{x} \in D)
\label{eq_add_op}
\end{equation}
\noindent
where $v_{i} (\mathbf{x}) := (a_{1}, \ldots, a_{i-1}, \underbrace{x_{i}}_{\text{ith coordinate}}, a_{i+1}, \ldots, a_{d})$, 
and $v_{d+1}(\mathbf{x}) := \mathbf{a}$.
\end{remark}

\begin{corollary}
A centred random field $Z$ possessing a covariance kernel $k$ has additive paths (up to a modification) if and only if $k$ is a positive definite kernel of the form
\begin{equation}
k(\mathbf{x}, \mathbf{x}') =  \sum_{i=1}^{d} \sum_{j=1}^{d} k_{ij}(x_{i}, x_{j}')
\label{additive_dependent}
\end{equation}
\end{corollary}

\begin{proof}
\noindent
If $Z$ has additive paths up to a modification, there exists a random field $(A_{\mathbf{x}})_{\mathbf{x} \in D}$ with additive 
paths such that $\forall \mathbf{x} \in D \ \mathbb{P}(Z_{\mathbf{x}}=A_{\mathbf{x}})=1$, and so $Z$ and $A$ have the same covariance kernel. 
Now, $A$ having additive paths, Remark \ref{rmk:additivefun} implies that 
$A_{\mathbf{x}} = \sum_{i=1}^{d} A_{v_{i}(\mathbf{x})} - (d-1)A_{v_{d+1}(\mathbf{x})} = \sum_{i=1}^{d} A^{i}_{\mathbf{x}_{i}}$, 
where $A^{i}_{\mathbf{x}_{i}} := A_{v_{i}(\mathbf{x})} - \frac{(d-1)}{d}A_{v_{d+1}(\mathbf{x})}$, 
so Equation \ref{additive_dependent} holds with $k_{ij}(\mathbf{x}_{i}, \mathbf{x}_{j}') := \text{cov}(A^{i}_{\mathbf{x}_{i}}, A^{j}_{\mathbf{x}_{j}'})$.
Reciprocally, from Proposition $3$, we know that it suffices for $Z$ to have additive paths that $k(\cdot,\mathbf{x}')$ is additive $\forall \mathbf{x}' \in D$. 
For a kernel $k$ such as in Eq. \ref{additive_dependent} and an arbitrary $ \mathbf{x}' \in D$, setting 
$$\forall \mathbf{x}_{i} \in D_{i}, \ \widetilde{k}_{i}(\mathbf{x}_{i}, \mathbf{x}') := 
\sum\limits_{\substack{j=1 }}^{d} k_{ij}(\mathbf{x}_{i}, \mathbf{x}_{j}') \ \ \ (1\leq i \leq d)$$ 
\noindent
we get $k(\mathbf{x},\mathbf{x}') = \sum_{i=1}^{d} \widetilde{k}_{i}(\mathbf{x}_{i}, \mathbf{x}') \ 
(\mathbf{x} \in D)$, so $k(\cdot,\mathbf{x}')$ is additive. 
\end{proof}

\begin{example}
Let us consider the following kernel over $\mathbb{R}^d \times \mathbb{R}^d$:
\begin{equation}
k(\mathbf{x},\mathbf{y}) = \sum_{i,j=1}^d \int_\mathbb{R} \kappa_i (x_i - u) \kappa_j (y_j - u) \mathrm{d}u  
\end{equation}
where the $\kappa_i$ are smoothing kernels over $\mathbb{R}$. Previous results on vector-valued random fields ensure that $k$ is a valid covariance function \cite{fricker2012multivariate}. Furthermore, the structure of $k$ corresponds to an additive kernel in the sense of Equation~\ref{additive_dependent}. According to Corollary~\ref{additive_dependent}, a random field with such kernel has additive paths (up to a modification), with univariate marginals  exhibiting possible cross-correlations.
\end{example}

\section{Extension to further operators. Focus on the Gaussian case}
\label{inv_Gaussian_case}

Composition operators constitute a remarkable class of linear maps since they can be defined on function spaces without any restriction. 
In particular, they similarly apply to random field paths or to kernel functions (with one argument fixed), 
so that taking out of a covariance a (combination of) composition(s) applied to a random field
and turning it into a (combination of) composition(s) on the covariance kernel appears as a natural operation. 

For more general classes of operators, however, operators on paths and operators on kernels are two different mathematical objects: It is a priori not obvious how to transform operators on paths into operators on the kernel, and even less straightforward to know when and how it is possible to define an operator on paths corresponding to a given operator on the kernel space. 

Given a linear operator $T:\mathbb{R}^{D} \to \mathbb{R}^{D}$ and a second-order centred process $Z$ such that $T(Z)$ is second order, 
generalizing the approach that lead to Prop.~\ref{prop:kernelProcess} enables us to characterize pathwise invariances of $Z$ by $T$ relying on second-order properties of the joint process $(Z_\mathbf{x}, T(Z)_\mathbf{x})_{\mathbf{x} \in D}$, without any additional assumption concerning $Z$'s probability distribution:

\begin{proposition} 
$Z=T(Z)$ up to a modification if and only if 
\begin{equation}
\label{eq_toto}
k(\mathbf{x},\mathbf{x}) = 2 \operatorname{cov}(T(Z)_\mathbf{x}, Z_\mathbf{x}) - \operatorname{var}(T(Z)_\mathbf{x}) 
\ \ \ \ (\mathbf{x} \in D ).
\end{equation}
\begin{proof}
Under the square-integrability and zero-mean hypotheses on $Z$ and $T(Z)$, $\operatorname{var}(T(Z)_\mathbf{x} - Z_\mathbf{x})=0$ is equivalent to $\mathbb{P}(T(Z)_\mathbf{x} = Z_\mathbf{x})=1$. 
\end{proof}
\end{proposition}

In the particular case of combinations of composition operators covered by Prop.~\ref{prop:kernelProcess}, it was possible to take $T$ out of the covariance and variance in the right hand side of Eq.~\ref{eq_toto}, leading to a further equivalence between pathwise invariance of $Z$ and invariance of $k$ under $T$. 
In greater generality, however, it is not straightforward how $T$ can be taken out of terms such as $\operatorname{cov}(T(Z)_\mathbf{x}, Z_\mathbf{x})$. 

We show in Section~\ref{subsec:RKHS} below that in case $Z$ is Gaussian and $T$ satisfies some technical condition with respect to $Z$, there exists an operator $\mathcal{T}$ defined over the Reproducing Kernel Hilbert Space associated with $Z$, such that $\text{cov}(T(Z)_{\mathbf{x}}, Z_{\mathbf{x'}}) = 
\mathcal{T}( 
\text{cov}(Z_{\cdot}, Z_{\mathbf{x'}})  
)(\mathbf{x})$, for all $\mathbf{x}, \mathbf{x'} \in D$. 

This construction based on the celebrated \textit{Lo\`eve isometry} \cite{ber:tho04} then enables us extending Prop. \ref{prop:kernelProcess} to a broader class of operators. 

Numerical examples are presented throughout the current section, including simulated paths of Gaussian random fields with argumentwise 
invariant kernels under various (integral and differential) operators, that subsequently serve as a basis to original applications in Gaussian Process regression, presented in Section \ref{sec:gpr}.

\subsection{Operating on the kernel via operators on paths}

We focus here on a centred Gaussian random field $(Z_{\mathbf{x}})_{\mathbf{x}\in D}$ defined over a compact set $D\subset \mathbb{R}^{d}$, 
with covariance kernel $k: D\times D \longrightarrow \mathbb{R}$. $k$ is here assumed continuous, so that the paths of 
$(Z_{\mathbf{x}})_{\mathbf{x}\in D}$ belong to some subspace of the space $\mathcal{B}$  of continuous functions over $D$, and are in particular square-integrable (with respect to Lebesgue's measure on $D$, say) by compacity of $D$.  
Let us further consider a linear map 
$T: \mathbb{R}^{D} \rightarrow \mathbb{R}^{D}$ 
acting on the paths of $Z$ and such that $T(Z)_{\mathbf{x}}$ is centred and square-integrable for all $\mathbf{x}\in D$. 

\medskip

In Proposition~\ref{prop:Tronde} below, the so-called \textit{Lo\`eve isometry} \cite{ber:tho04} allows us to define an operator, derived from $T$, acting on the Reproducing Kernel Hilbert Space (RKHS) associated with $k$. Let us first recall some useful definitions and the isometry in question. 
The RKHS $\mathcal{H}$ associated with $k$ \cite{aro50} can be defined as functional completion of the function space spanned by the $k(\cdot, \mathbf{x})$'s: 
$$\mathcal{H}=\overline{\text{span}(k(\cdot, \mathbf{x}), \mathbf{x} \in D)}$$ 
equipped with the scalar product defined by $\langle k(\cdot, \mathbf{x}), k(\cdot, \mathbf{x'}) \rangle_{\mathcal{H}} 
= k(\mathbf{x}, \mathbf{x'})$.
A crucial state-of-the-art result is that $\mathcal{H}$ is isometric to the Hilbert space generated by the random field 
$Z$ \cite{ber:tho04}: 
$$\mathcal{L}(Z) = \overline{\text{span}(Z_{\mathbf{x}}, \mathbf{x} \in D)},$$
where the adherence is taken with respect to the usual $L^{2}(\mathbb{P})$ topology on the space of (equivalence classes of) square-integrable random variables. 

\begin{proposition} (Lo\`eve isometry) 
\label{loeve}
The map $\Psi : \mathcal{H} \rightarrow \mathcal{L}(Z)$ defined by:
$$ k(\cdot, \mathbf{x}) \rightarrow Z_\mathbf{x}$$
for all $\mathbf{x} \in D$ and extended by linearity and continuity, is an isometry from $(\mathcal{L}(Z), \langle \cdot{}, \cdot{}\rangle_{L^2})$ 
to $(\mathcal{H}, \langle \cdot{}, \cdot{}\rangle_{\mathcal{H}})$.  
\end{proposition}

As shown below, the Lo\`eve isometry allows us to link operators on the paths of $Z$ to corresponding operators on the RKHS, provided that the random variables $T(Z)_\mathbf{x}$ ($\mathbf{x} \in D$) belong to $\mathcal{L}(Z)$:

\begin{proposition}
\label{prop:Tronde}
Let $T: \mathcal{B} \rightarrow \mathbb{R}^{D}$ be such that for any $\mathbf{x}\in D$, $T(Z)_{\mathbf{x}} \in \mathcal{L}(Z)$.
Then, there exists a unique operator $\mathcal{T}:\mathcal{H} \rightarrow \mathbb{R}^{D}$ satisfying
\begin{equation}
\label{eq:Tronde}
\operatorname{cov}( T(Z)_{\mathbf{x}}, Z_{\mathbf{x'}})=\mathcal{T}(k(.,\mathbf{x'})) (\mathbf{x})
\ \ \ \ (\mathbf{x},\mathbf{x'}\in D)
\end{equation}
and such that $ \mathcal{T}(h_n)(\mathbf{x}) \longrightarrow \mathcal{T}(h)(\mathbf{x})$ for all $\mathbf{x} \in D$  and $h_n \stackrel{\mathcal{H}}{\longrightarrow} h$.
\end{proposition}
\begin{proof} 
Let $\mathcal{T}:\mathcal{H} \rightarrow \mathbb{R}^{D}$ be an operator satisfying (\ref{eq:Tronde}) and the pointwise convergence condition. Since $Z_{\mathbf{x'}} = \Psi (k(., \mathbf{x'}))$, we have:
$$
\mathcal{T}(k(.,\mathbf{x'})) (\mathbf{x}) = \operatorname{cov}( T(Z)_{\mathbf{x}}, \Psi(k(., \mathbf{x'})))
\ \ \ \ (\mathbf{x}, \mathbf{x'} \in D)
$$
This is immediately extended in a unique way to $\mathcal{H}$ by linearity and continuity of the isometry $\Psi$, leading to:
\begin{equation}
\label{eq:Tronde_H}
\mathcal{T}(h) (\mathbf{x}) = \operatorname{cov}( T(Z)_{\mathbf{x}}, \Psi(h) ) 
\ \ \ \
(\mathbf{x} \in D, h \in \mathcal{H}). 
\end{equation}
Conversely, using again properties of $\Psi$, one easily checks that (\ref{eq:Tronde_H}) defines a linear map satisfying (\ref{eq:Tronde}) and the pointwise convergence condition.
\end{proof} 

The construction proposed above will serve as basis for an invariance result, given in Prop.~\ref{prop:invRKHS}. 
Before stating it, let us examine and discuss in more detail the assumptions made in Prop.~\ref{prop:Tronde} and the relation between $T$ and $\mathcal{T}$, both through examples and analytical considerations. 

\begin{example}
\label{ex:TrondeOpLinCoCoop}
Let $T$ be a linear combination of composition operators, $T = \sum_{i=1}^q \alpha_{i} T_{v_i}$, such as introduced in Def.~\ref{def:comofcom}.
Recall that $T$ similarly applies to random field paths or to kernel functions, with
$T(Z)_{\cdot}= \sum_{i=1}^q \alpha_{i} Z_{v_i(\cdot)}$ and $T(k(\cdot,\mathbf{x'})) = \sum_{i=1}^q \alpha_{i} k(v_i(\cdot), \mathbf{x'})$. In particular, we directly obtain that the condition $T(Z)_{\mathbf{x}} \in \mathcal{L}(Z)$ 
is fulfilled, so that Prop.~\ref{prop:Tronde} can be applied.

It is then easy to check that $\mathcal{T}(k(\cdot,\mathbf{x'})) = T(k(\cdot, \mathbf{x'}))$, 
implying that $\mathcal{T}$ is the unique representer of $T$ on $\mathcal{H}$ satisfying the pointwise convergence condition of Prop.~\ref{prop:Tronde}. In other words, here 
$\mathcal{T}=T_{\vert \mathcal{H}}$. 

Note that this example also illustrates that 
$\mathcal{T}(\mathcal{H})$ may differ from $\mathcal{H}$. 
Indeed, choosing a composition operator $T = T_{v}$ and fixing $\mathbf{x'} \in D$, we have
$$\mathcal{T}(k(.,\mathbf{x'}))(\mathbf{x}) = cov(Z_{v}(\mathbf{x}), Z_{\mathbf{x'}})=k(v(\mathbf{x}), \mathbf{x'}),$$
and so $\mathcal{T}(k(.,\mathbf{x'}))=k(v(.), \mathbf{x'})$. 
Taking for instance the $1$-dimensional RKHS $\mathcal{H}$  spanned by the 1st order polynomial $e_1(t)=t$ on $D=[0,1]$ (with kernel $k(t,t') = e_1(t)e_1(t') = tt'$) and choosing $v(t)=t^2$, we see that  
$\mathcal{T}(k(.,t')): t \rightarrow t^2 t'$ is a second order polynomial, and thus not in $\mathcal{H}$.
\end{example}

\begin{example}
\label{ex:TrondeOpInt}
Let us now consider a measure $\nu$ on $D$ 
such that 
$$\int_D \sqrt{k(\mathbf{u},\mathbf{u})} 
\mathrm{d}\nu(\mathbf{u}) < + \infty, $$
\noindent
and define $T(Z)_{\mathbf{x}}= \int_{D} Z_\mathbf{u} d\nu(\mathbf{u})$ for all $\mathbf{x} \in D$.
Then, relying on the Fubini-Tonelli theorem,  
$\mathcal{T}(h) = \int_D h(\mathbf{u}) d \nu(\mathbf{u})$. In other words, $\mathcal{T}=T_{\vert \mathcal{H}}$ again. 
\end{example}

\subsection{A detour through the Karhunen-Lo\`eve expansion}

In both Examples \ref{ex:TrondeOpLinCoCoop} 
and \ref{ex:TrondeOpInt}, we found out that $\mathcal{T}=T_{\vert \mathcal{H}}$. From this, we may wonder under which circumstances $\mathcal{T}$ is a restriction of $T$. 
The spectral framework, and more specifically the \textit{Karhunen-Lo\`eve (KL) expansion}, is a suitable setting 
to investigate such question. As a preliminary to a sufficient condition for $\mathcal{T}=T_{\vert \mathcal{H}}$ to hold,  
let us recall some useful basics concerning the KL expansion.

\medskip

In a nutshell, the starting point of KL is the \textit{Mercer decomposition} (See \cite{mer1909}, with generalizations in \cite{koenig86,ste:sco2012}) 
of the continuous covariance kernel $k$:  
Given any finite measure $\nu$ on the Borel algebra of $D$ whose support is $D$ (typically the Lebesgue measure $\lambda$), there exists an orthonormal basis 
$(\varphi_n)_{n \geq 1}$ of $L^2(\nu)$
and a sequence of non-negative real numbers 
$(\gamma_n)_{n \geq 1} \downarrow 0$ such that:
\begin{equation}
\label{eq:KL_kernel}
k(\mathbf{x},\mathbf{x'}) = \sum_{n=1}^{+\infty} \gamma_{n} \varphi_n(\mathbf{x}) \varphi_n(\mathbf{x'}),
\end{equation}
\noindent
where the convergence is absolute and uniform on $D$. 
Note that the finite trace hypothesis $\int_{D} k(\mathbf{u}, \mathbf{u}) \mathrm{d}\nu(\mathbf{u}) < + \infty $ often given as prerequisite of the Mercer theorem is automatically fulfilled here, 
considering the assumptions made on $k$. 

Relying on Eq.~\ref{eq:KL_kernel}, it is well-known (See, e.g., \cite{ber:tho04}) that the RKHS  $\mathcal{H}$ can then be represented as a subspace of $L^2(\nu)$, in the following way:
\begin{equation}
\label{eq:H_with_KL}
\mathcal{H} = \left\{ \mathbf{x} \rightarrow f(\mathbf{x}) = \sum_{n=1}^{+\infty} \alpha_n \varphi_n(\mathbf{x}), 
\quad s.t. \quad \sum_{n=1}^{+\infty} \frac{\alpha_n^2}{\gamma_{n}} < +\infty \right\}. 
\end{equation}
Furthermore, relying on the Lo\`eve isometry (Cf. Prop.~\ref{loeve}), the random field $Z$ itself can be expanded with respect to the $\varphi_n$'s, leading to the KL expansion: 
\begin{equation}
\label{eq:KL_process}
Z_\mathbf{x} = \sum_{n=1}^{+\infty} \sqrt{\gamma_{n}} \zeta_n \varphi_n(\mathbf{x})
\end{equation}
where the $\zeta_n$'s are independent standard Gaussian random variables, and the series is uniformly convergent  with probability $1$ \cite{kue71}.
In particular, noting that 
$\mathbb{E}\left[ \int_{D} 
Z_{\mathbf{u}}^2 \mathrm{d}\nu (\mathbf{u})
\right] 
= \int_{D} 
k(\mathbf{u},\mathbf{u}) \mathrm{d}\nu (\mathbf{u}) 
< + \infty$, we get (with probability $1$) both that the paths of $Z$ are in $L^2(\nu)$ and that the series of Eq.~\ref{eq:KL_process} converges normally. Consequently, in case of a bounded operator $T$ from $L^2(\nu)$ to itself, 
\begin{equation}
\label{eq:KL_image}
T(Z)_{\cdot} = \sum_{n=1}^{+\infty} \sqrt{\gamma_{n}} \zeta_n T(\varphi_n)(\cdot)
\end{equation}
with probability $1$, where the convergence is normal. 
%
Note that in cases such as the one of the differentiation operator (See, e.g., \nocite{kad1967} for a differentiation of the KL expansion), $T$ is not bounded with respect to the usual $L^2(\nu)$ norm, but similar normal convergence results may be obtained by considering a source space of differentiable elements equipped with an \textit{ad hoc} topology (e.g., Sobolev spaces). 

\medskip

Concerning our question on operators on paths vs on the RKHS, we obtain by substituting $Z$ and $T(Z)$ by their respective expansions in $\mathcal{T}$'s definition:  
\begin{equation}
\begin{split}
\mathcal{T}(k(\cdot, \mathbf{x'})) &= \operatorname{cov}(T(Z)_{\cdot}, Z_\mathbf{x'}) \\ 
&= \operatorname{cov} \left(\sum_{n=1}^{+\infty} \sqrt{\gamma_{n}} \zeta_n T(\varphi_n)(\cdot),
\sum_{m=1}^{+\infty} \sqrt{\gamma_m} \zeta_m \varphi_m (\mathbf{x'}) \right) \\
&= \sum_{n=1}^{+\infty} \gamma_{n} \varphi_n(\mathbf{x'}) T(\varphi_n)(\cdot).
\end{split}
\end{equation}
Now, using the Mercer decomposition (\ref{eq:KL_kernel})
of $k$ and the boundedness of $T$, 
$T(k(\cdot, \mathbf{x'})) =  \sum_{n=1}^{+\infty} \gamma_{n} \varphi_n(\mathbf{x'}) T(\varphi_n)(\cdot)$, so we conclude that $\mathcal{T}=T_{\vert \mathcal{H}}$. 
Besides, on may also notice that $\zeta_n \in \mathcal{L}(Z)$ since 
$\zeta_n = \frac{1}{\sqrt{\gamma_{n}}} \langle Z_{\cdot}, \varphi_n \rangle_{L^2(\nu)}.$
Using that $\mathbb{E}[||T(Z)_{\cdot}||_{L^2(\nu)}^2] < + \infty$, we finally also get that $T(Z)_\mathbf{x}\in \mathcal{L}(Z)$ $\nu$-a.e. 

\subsection{Invariances and Gaussian random fields}
\label{subsec:RKHS}

Coming back to invariances, we now give a characterization result, that generalizes those of Section $2$ in the particular case of Gaussian fields:

\begin{proposition}
\label{prop:invRKHS}
Under the assumptions of Prop.~\ref{prop:Tronde},
the three following conditions are equivalent:
\begin{itemize}
\item[(i)] $Z=T(Z)$ \textit{up to a modification} 
\item[(ii)] $\mathcal{T}(k(\cdot, \mathbf{x'}))=k(\cdot, \mathbf{x'}) \ (\forall \mathbf{x'} \in D) $
\item[(iii)] $\mathcal{T}=\text{Id}_{\mathcal{H}}$
\end{itemize}
\end{proposition}

\begin{proof} By the pointwise convergence condition on $\mathcal{T}$, (ii) and (iii) are equivalent.
Now, let us prove the equivalence between (i) and (ii). 
Since for any arbitrary $\mathbf{x} \in D$, $Z_{\mathbf{x}}-T(Z)_{\mathbf{x} } \in \mathcal{L}(Z)$, 
we have 
by duality:
\begin{align*}
Z_{\mathbf{x}} \stackrel{a.s.}{=} T(Z)_{\mathbf{x}} \Longleftrightarrow & \ Z_{\mathbf{x}} - T(Z)_{\mathbf{x}} \stackrel{a.s.}{=} 0\\
\Longleftrightarrow & \ \text{cov}( Z_{\mathbf{x}} - T(Z)_{\mathbf{x}}, Z_{\mathbf{x'}})= 0 \ \forall \mathbf{x'} \in D \\
\Longleftrightarrow & \ k(\mathbf{x}, \mathbf{x'}) = \mathcal{T}( k(\cdot, \mathbf{x'}) ) (\mathbf{x}) \ \forall \mathbf{x'} \in D 
\end{align*}
\end{proof} 

Proposition \ref{prop:invRKHS} can be used to define families of centred Gaussian field models satisfying linear-type properties, simply by looking at their kernel. This includes for instance the case of Gaussian random fields with centred paths (or \textit{mean-centered} fields, to use the terminology of \cite{deh2007}) 
and Gaussian random fields whose paths are solutions of linear differential equations, as illustrated below (See also \cite{sch:sch2012} for recent results on vector fields with divergence-free and curl-free paths).

\noindent
\begin{example}[Gaussian random fields with centred paths]
Let $\nu$ be a probability measure on $D \subset \mathbb{R}^d$
such that $\int_D \sqrt{k(\mathbf{u},\mathbf{u})} d\nu(\mathbf{u}) < + \infty$.
Then $Z$ has centered paths -- i.e. $\int_D Z_\mathbf{u} \mathrm{d} \nu(\mathbf{u}) = 0$ --
 if and only if $\int_D k(\mathbf{x},\mathbf{u}) \mathrm{d} \nu(\mathbf{u}) = 0, \forall \mathbf{x} \in D$. Indeed, define $T$ by $T(Z)_\mathbf{x} = Z_{\mathbf{x}} - \int_D Z_{\mathbf{u}} \mathrm{d} \nu(\mathbf{u})$ for all $\mathbf{x} \in D$.
Following Example \ref{ex:TrondeOpInt}, we have $\mathcal{T}(h) = h - \int_D h(\mathbf{u}) \mathrm{d} \nu(\mathbf{u})$, and the result comes from Proposition \ref{prop:invRKHS}. 

For instance, the kernel $k_{0}$ defined by 
\begin{equation}
\begin{split}
k_0(\mathbf{x},\mathbf{y}) = 
k(\mathbf{x},\mathbf{y}) & - \int k(\mathbf{x},\mathbf{u}) \mathrm{d} \nu(\mathbf{u}) \\ 
& - \int k(\mathbf{y},\mathbf{u}) \mathrm{d} \nu(\mathbf{u}) + \int k(\mathbf{u},\mathbf{v})  \mathrm{d} \nu(\mathbf{u})  \mathrm{d} \nu(\mathbf{v}) 
\label{eq:k0}
\end{split}
\end{equation}
satisfies the above condition. Figure~\ref{fig:ex_invpath} (a) shows some sample paths of a centred Gaussian random process based possessing a kernel of that form.
\end{example}

\begin{example}
We illustrate here the case where the sample paths of a Gaussian process are solution to the differential equation:
\begin{equation}
y'' (t) + y(t) = 2t.
\label{eq:DE}
\end{equation}
The solutions of the associated homogeneous equation are the functions satisfying $y = - y''$ so they correspond to invariant functions with respect to $T:\ f \rightarrow - f''$. The solutions of the homogeneous equation are well-known to be in $\mathrm{span}(\cos, \sin)$ which can be endowed with the following kernel
\begin{equation}
k_{ode}(s,t)= (\cos(s)\ \sin(t)) \; \Sigma \: (\cos(s)\ \sin(t) )^{t} 
\label{eq:ODEkern}
\end{equation} 
where $\Sigma$ is a symmetric positive semi-definite $2 \times 2$ matrix. $k_{ode}(s,\cdot)$ is solution to the homogeneous equation (i.e. $k_{ode}(s,\cdot)$ is $T$-invariant) for all $s$ so the sample paths of a centred Gaussian process with such kernel inherit this property. 
Let $Y$ be a Gaussian process with mean $\mu(t) = 2t$ and covariance $k$. Since $\mu$ is a particular solution of Eq.~\ref{eq:DE}, $Y$ has sample paths satisfying this differential equation. This is illustrated in Figure~\ref{fig:ex_invpath}.b.
\end{example}

\begin{figure}
\centering
\begin{subfigure}[b]{0.3\textwidth}
    \centering
    \includegraphics[width=\textwidth]{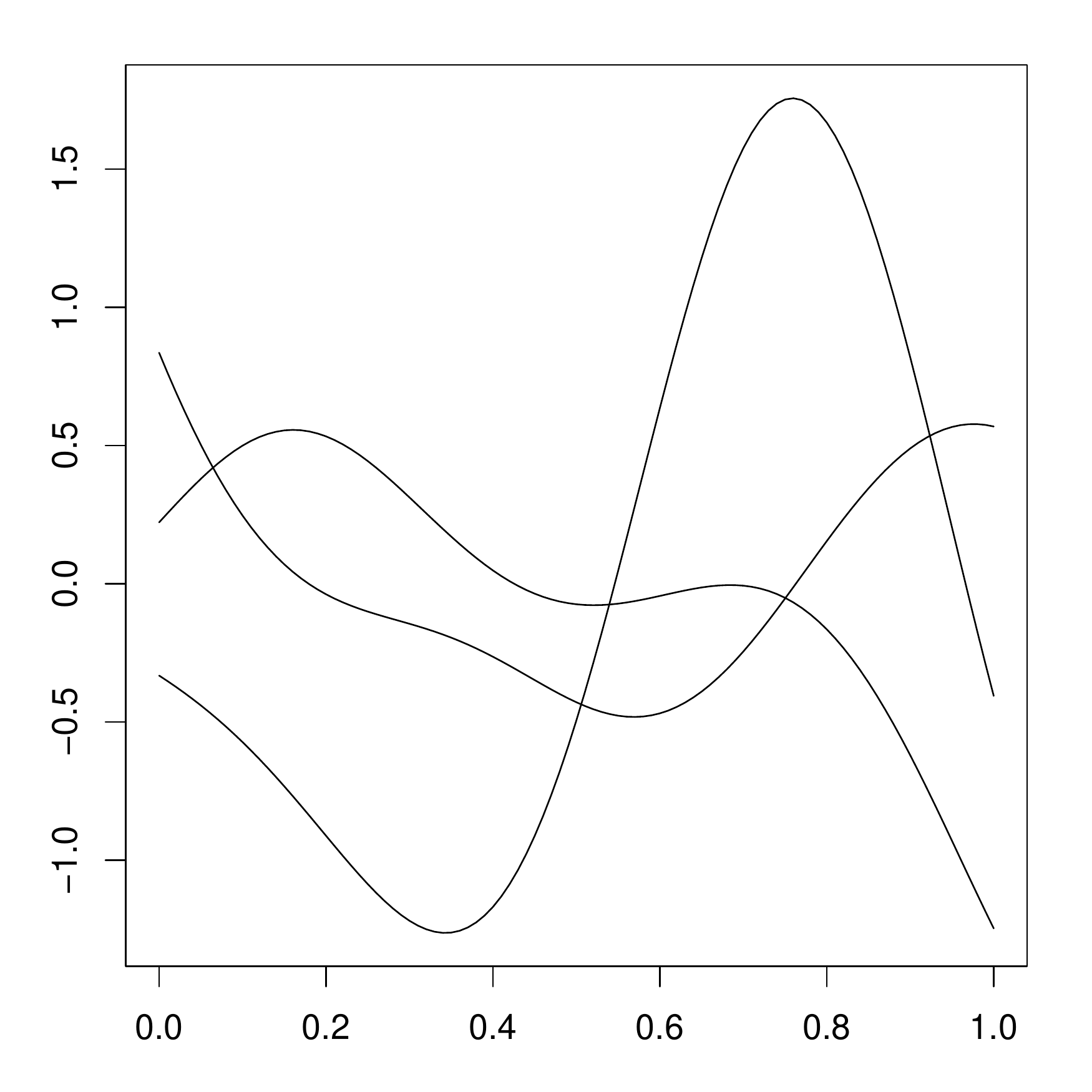}
    \caption{Centred sample paths, based on $k_0$. \quad \quad \quad}
\end{subfigure} \quad
\begin{subfigure}[b]{0.3\textwidth}
    \centering
    \includegraphics[width=\textwidth]{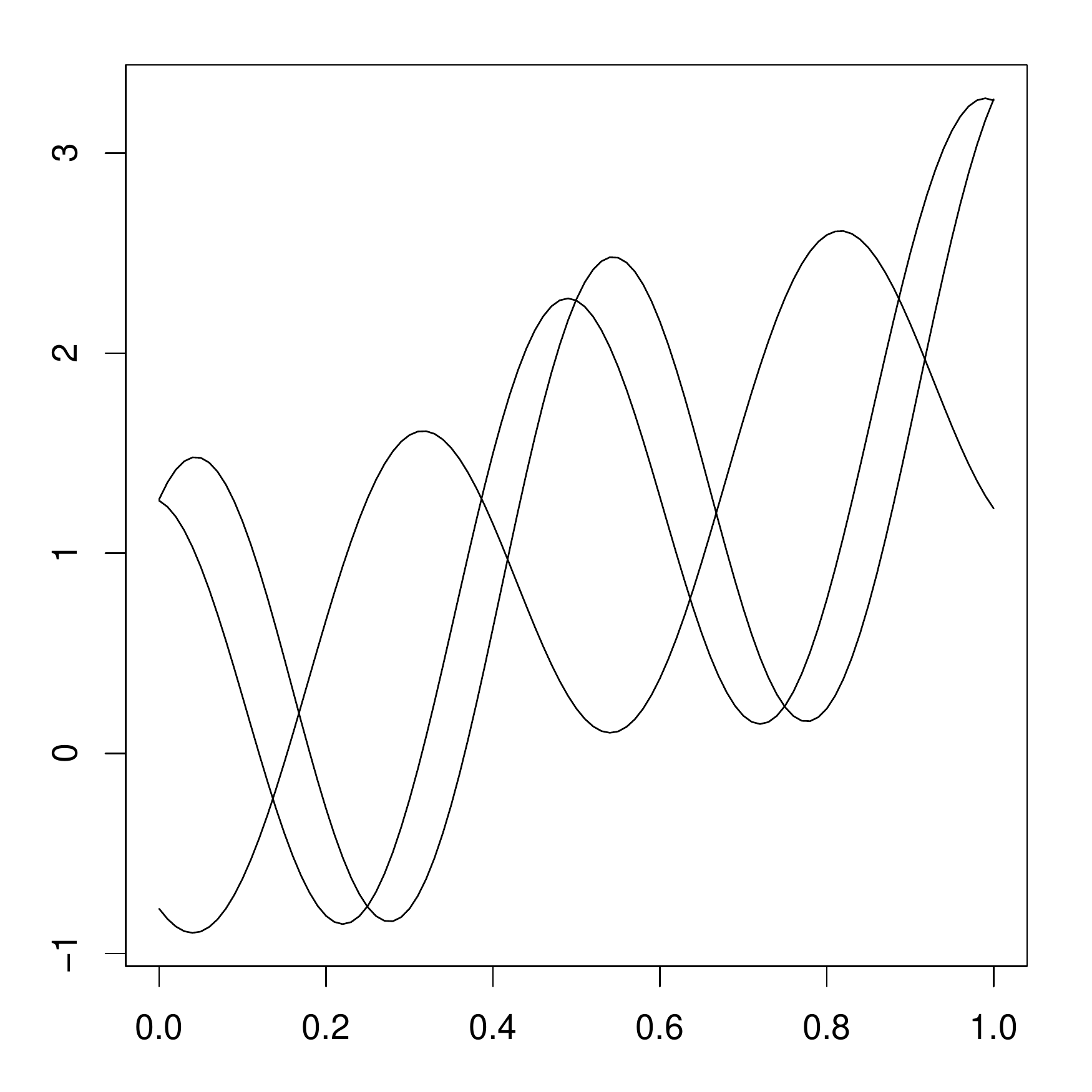}
    \caption{Random solutions to Eq.~\ref{eq:DE}, based on $k_{ode}$.}
\end{subfigure} \quad
\begin{subfigure}[b]{0.3\textwidth}
    \centering
    \includegraphics[width=\textwidth]{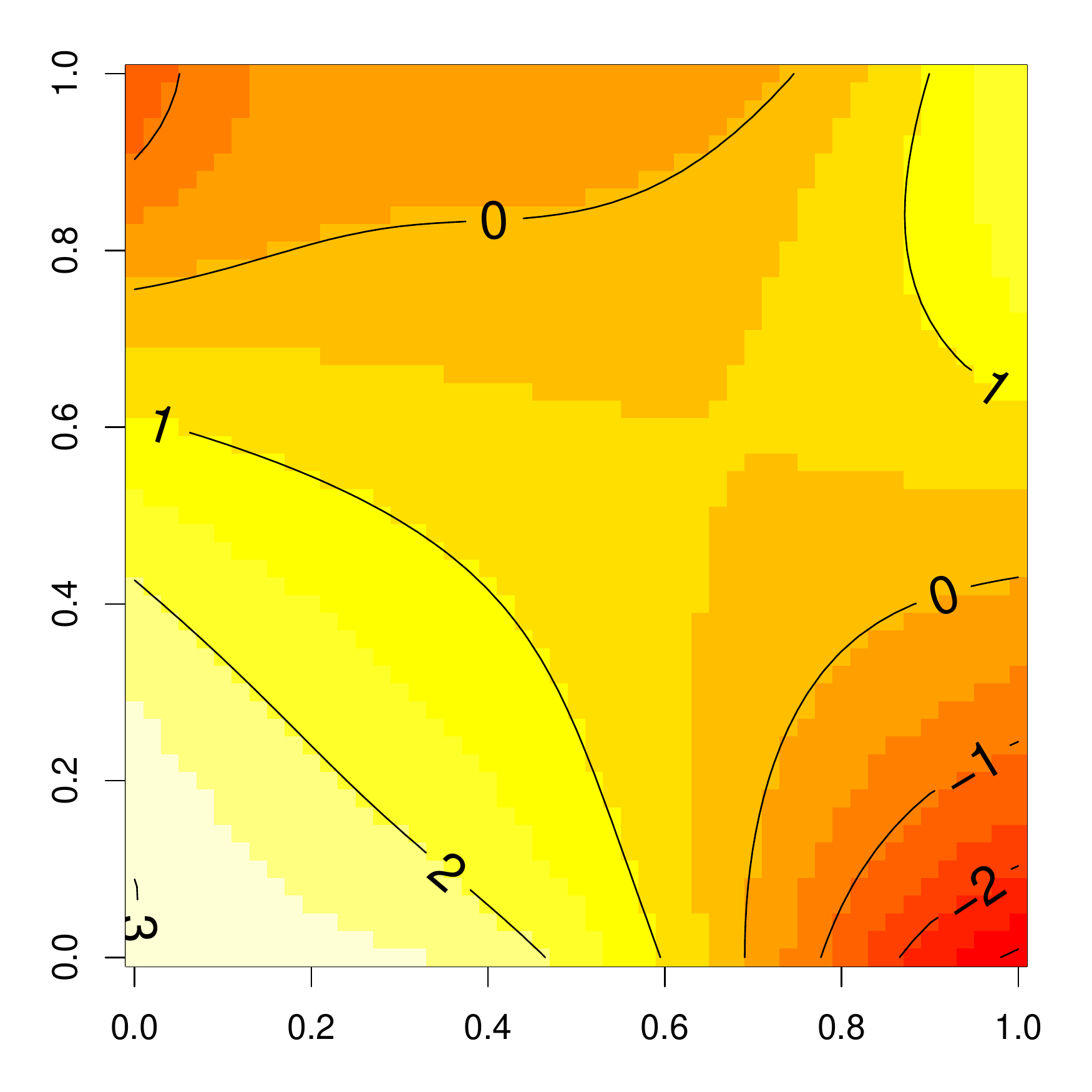}
    \caption{Harmonic sample path, based on $k_{harm}$.}
\end{subfigure}%
\caption{Examples of sample paths invariant under various operators.}
\label{fig:ex_invpath}
\end{figure}

\begin{example}
In the previous example, the solutions of the ODE belong to a 2-dimensional space. We consider here another ODE, the \textit{Laplace equation} $\Delta f = 0$, for which the space of solutions is infinite dimensional. 
The solutions to this equation are called harmonic functions and we will call harmonic kernels any positive definite function satisfying the ODE argumentwise: $(\Delta k(\cdot,\mathbf{x'})) = 0$ 
$(\mathbf{x'} \in D)$. 
Examples of such harmonic kernels can be found in the recent literature (See respectively~\cite{schaback2009solving,hon2013solving} for 2D and 3D input spaces). We will focus here on the following kernel over $\mathbb{R}^2 \times \mathbb{R}^2$:
\begin{equation}
	k_{harm}(\mathbf{x},\mathbf{y}) = \exp \left( \frac{x_1 y_1 + x_2 y_2}{\theta^2} \right) \ \cos \left( \frac{x_2 y_1 - x_1 y_2}{\theta^2} \right).
    \label{eq:kharm}
\end{equation}
Proposition~\ref{prop:invRKHS} can be applied to the operator $f \rightarrow f + \Delta f$ so the sample paths obtained with $k_{harm}$ also are harmomic functions. This can be seen in the right panel of Figure~\ref{fig:ex_invpath} where the sample path shows some special features of harmonic functions such as the absence of local minimum.

\end{example}

\section{Applications in Gaussian process regression}
\label{sec:gpr}

The aim of this section is to discuss and illustrate the use of argumentwise invariant kernels in Gaussian process regression (GPR).
The main idea behind this approach is to incorporate invariance assumptions within GPR. 
As we will see, using such kernels can significantly improve the predictivity of GPR in cases where \textit{structural priors} on the function to approximate, involving invariances under bounded linear operators, are available.  

\medskip

GPR gives a very convenient stochastic framework for modelling a function $f: D \rightarrow \mathbb{R}$ based on a finite set of observations $y_i = f(\mathbf{x}^{(i)})$, $1 \leq i \leq n$ and a Gaussian process prior on $f$. 
The literature of GPR and related methods is scattered over several fields including statistics and geostatistics~\cite{ste99}, machine learning~\cite{ras:wil06} and functional analysis~\cite{ber:tho04}. 
Depending on the scientific community, the predictor of $f$ is either defined as best linear unbiased predictor of a square-integrable (or intrinsic) random field, conditional expectation of a Gaussian Process, or interpolator with minimal norm in RKHS settings. 
One striking fact is that, given any positive definite kernel $k$, the approaches end up with the same expressions for the best predictor and for the ''conditional'' kernel describing the remaining uncertainty on $f$: 

\begin{equation}
\begin{split}
m(\mathbf{x}) &= \mathbf{k}(\mathbf{x})^t K^{-1} \mathbf{Y} \\
c(\mathbf{x},\mathbf{x}') &= k(\mathbf{x},\mathbf{x'}) - \mathbf{k}(\mathbf{x})^t K^{-1} \mathbf{k}(\mathbf{x}')
\end{split}
\label{eq:GPR}
\end{equation}
where $\mathbf{k}(\mathbf{x})= (k(\mathbf{x},\mathbf{x}^{(i)}))_{1 \leq i \leq n}$ and $K=(k(\mathbf{x}^{(i)},\mathbf{x}^{(j)}))_{1 \leq i,j \leq n}$. 
We will discuss in the next section the influence of using invariant kernel in such models.

\subsection{Gaussian Process Regression with invariant kernels}
We consider here a bounded operator $T$ on the paths as defined in Section~\ref{inv_Gaussian_case}, and we use for convenience the same letter to denote T's restriction to the RKHS $\mathcal{H}$.
Assuming that $k(\mathbf{x},.) = T(k(\mathbf{x},.)) \ (\mathbf{x}\in D)$, it was already established that the paths of a centred Gaussian random field with kernel $k$ are invariant under $T$. 
We now establish further that both the GPR predictor and the conditional distribution of such random field knowing response values at a finite set of points are invariant as well. 
\begin{proposition}
Let $Z$ be a centred Gaussian field with argumentwise 
$T$-invariant kernel $k$ and $y_i = Z(\mathbf{x}^{(i)})$ $(1 \leq i \leq n)$ be a finite set of observations. Then, 
\begin{itemize}
 \item[(i)] The GPR predictor $m$ is $T$-invariant
 \item[(ii)] The ''conditional covariance kernel'' $c$ is argumentwise $T$-invariant
 \item[(iii)] $Z$ conditioned on the evaluation results is $T$-invariant, up to a modification. 
 Consequently, 
conditional simulations of $Z$ are $T$-invariant.
 \end{itemize}
\end{proposition}
\begin{proof}
The properties $(i)$ and $(ii)$ are a direct consequence of the linearity of $T$. For example, we have for $(ii)$:
\begin{equation}
\begin{split}
T(c(\mathbf{x},.))(\mathbf{x}') &= T(k(\mathbf{x},.) - \mathbf{k}(\mathbf{x})^t K^{-1} \mathbf{k}(.))(\mathbf{x}') \\
& = T(k(\mathbf{x},.))(\mathbf{x}') - \mathbf{k}(\mathbf{x})^t K^{-1} T(\mathbf{k}(.))(\mathbf{x}') \\
& = k(\mathbf{x},\mathbf{x'}) - \mathbf{k}(\mathbf{x})^t K^{-1} \mathbf{k}(\mathbf{x}').
\end{split}
\end{equation}
Furthermore the conditional distribution of $Z$ knowing evaluation results simplifies as $\mathcal{L}(Z | Z(\mathbf{X})=\mathbf{Y}) \equiv m + \mathcal{GRF}(0, c)$, 
where $\mathcal{GRF}(0, c)$ stands for the distribution of a centred Gaussian random field with covariance kernel $c$. 
%
According to Proposition~\ref{prop:invRKHS}, a random field with distribution $\mathcal{GRF}(0, c)$ is $T$-invariant up to a modification. 
$(iii)$ follows using the linearity of $T$. 
\end{proof} 

\subsection{Illustration on examples}
We now consider invariant kernels introduced in the examples of the previous sections and study associated GPR models and predictions. More precisely, we focus on $3$ case studies involving various priors: zero-mean functions, solutions to $y''(x) + y(x) = 2x$ and solutions to $ \Delta y (\mathbf{x})= 0$.

\paragraph{GPR with centred paths}
Here we assume that $D = [-\pi,\pi]$ and that the function to approximate is $f(x)= \cos(x) + \cos(2 x)+ \cos(3 x) + \sin(x/2)$.
Assuming that for some reason, it is known \textit{a priori} that the integral of $f$ over $D$ is zero, it is of particular interest to incorporate this knowledge into the model. 
To get an insight of how GPR is improved by incoporating this structural prior, 
we compare predictions based on the following kernels:
\begin{equation}
\begin{split}
k(t,t') &= \exp \left( -4 (t-t')^2 \right) \\
k_{inv}(t,t') &= k(t,t') - \int k(t,u) \mathrm{d}u - \int k(t',u) \mathrm{d}u + \iint k(u,v) \mathrm{d}u \mathrm{d}v.
\end{split}
\end{equation}
The integral of $k_{inv}$ with respect to any of its variables is zero, so the paths of the associated centred Gaussian Process will inherit this property.  
As a consequence, choosing a kernel such as $k_{inv}$ allows incorporating the prior information $\int f(u) \mathrm{d}u = 0 $ in GPR modeling, as illustrated in Figure~\ref{fig:zero-mean_GPR} (a).

\medskip  

In a second time, we assume that evaluation results $y_i=f(x^{(i)})$ at $12$ distinct points $x^{(i)} \in D$ are available, and we compare Gaussian Process conditional distributions based on both kernels. 
As seen in Figure~\ref{fig:zero-mean_GPR} the use of $k_{inv}$ improves considerably GPR predictions since $m$ recovers the large peak in the center of the domain. This is reflected by 
root integrated squared errors, with values of 0.04 and 1.06 for $k_{inv}$ and $k$, respectively.

\begin{figure}
\centering
\begin{subfigure}[b]{0.45\textwidth}
    \centering
    \includegraphics[width=\textwidth]{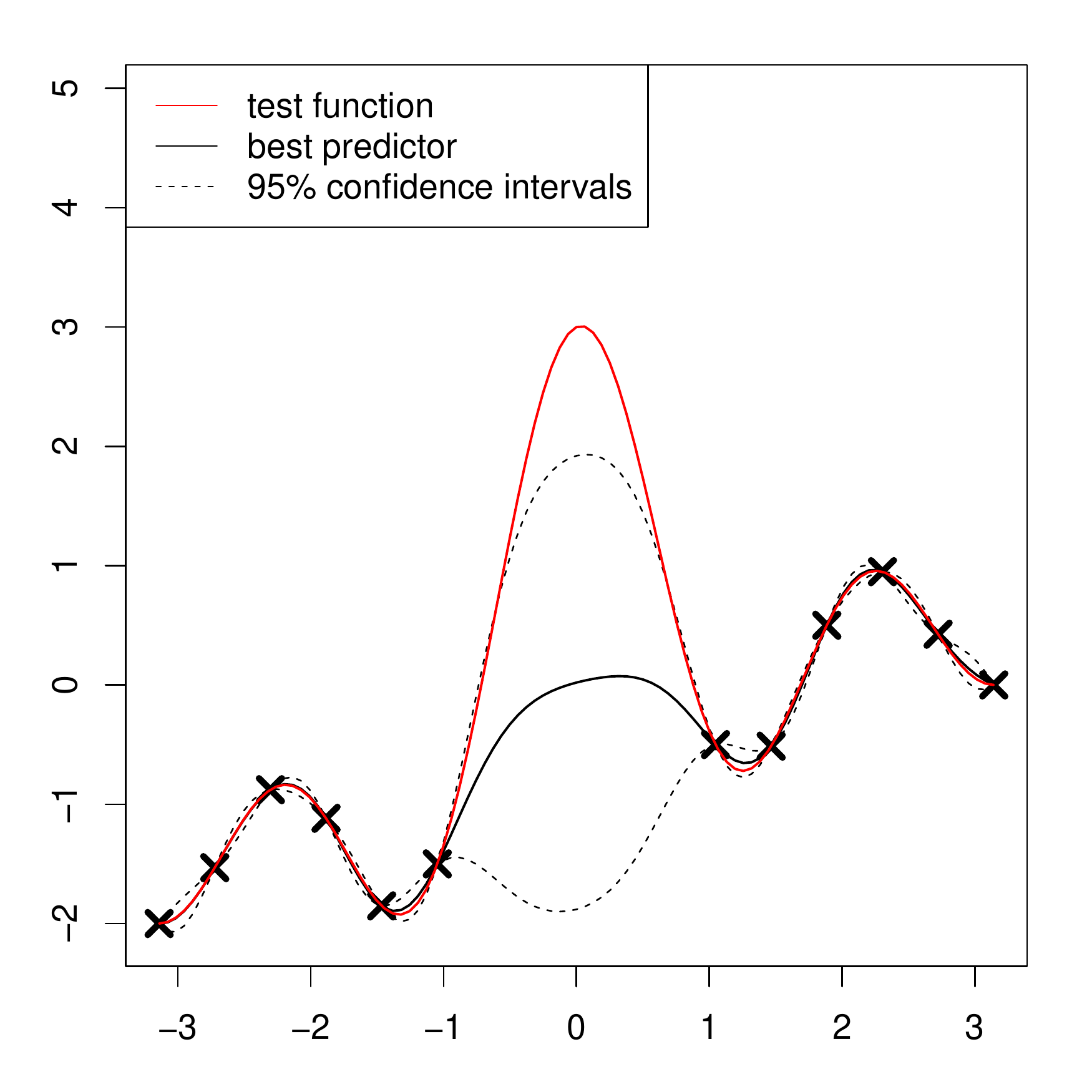}
    \caption{GPR with kernel $k$}
\end{subfigure} \qquad
\begin{subfigure}[b]{0.45\textwidth}
    \centering
    \includegraphics[width=\textwidth]{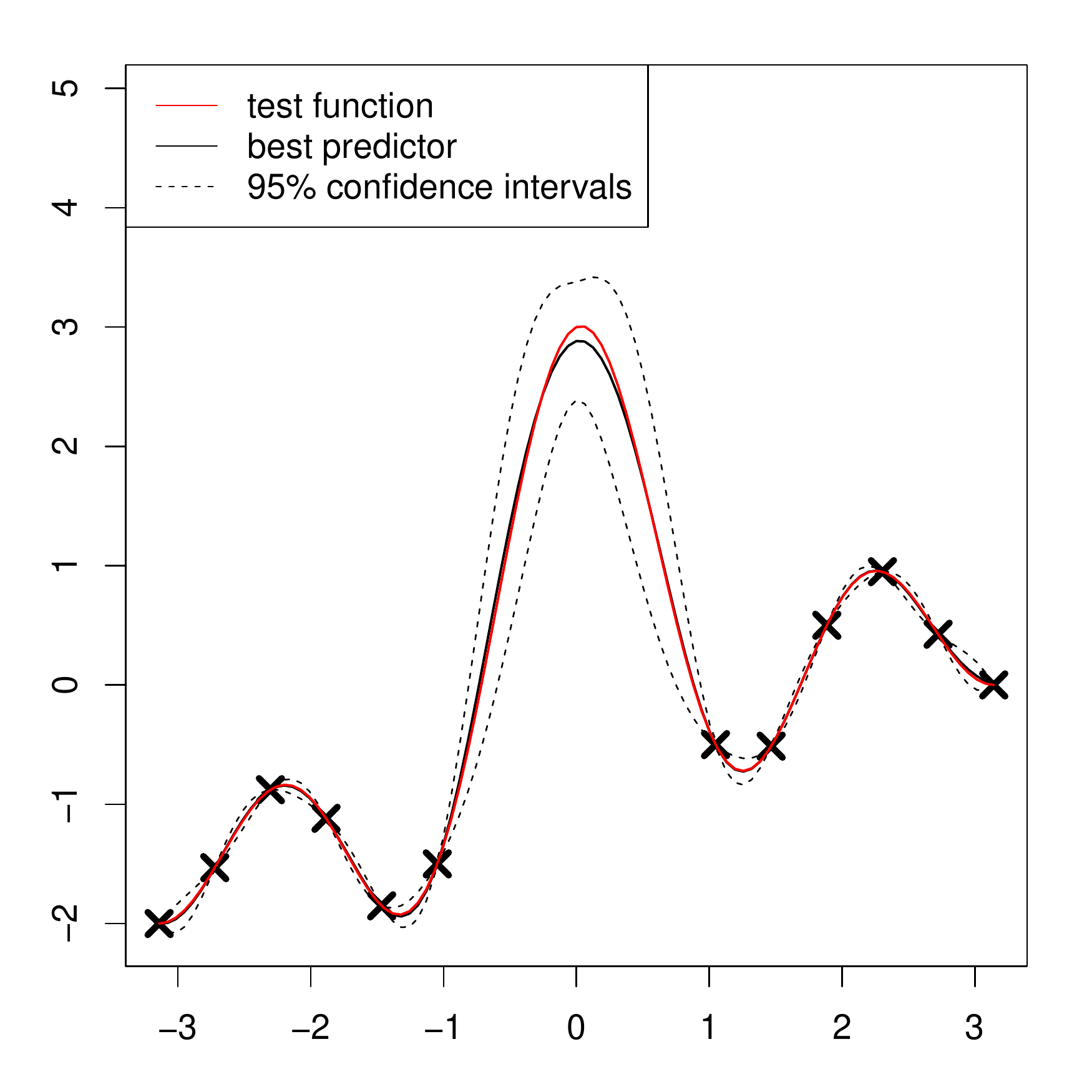}
    \caption{GPR with kernel $k_{inv}$}
\end{subfigure}%
\caption{Comparison of two GPR models. On the left panel, the model is based on an usual squared exponential kernel whereas on the right one it takes into account the zero-mean property of the function to approximate.}
\label{fig:zero-mean_GPR}
\end{figure}

\paragraph{GPR of a solution to a univariate linear ODE}
We saw in Section 3 that a Gaussian Process $Y$ with mean $\mu(x)=2x$ and kernel given by Eq.~\ref{eq:ODEkern} is equivalent to a process with paths satisfying the ODE $y''(x) + y(x) = 2x$. Figure~\ref{fig:ex_DEcond} shows the conditional distribution of $Y$ given evaluations at one or two points.
It can be seen on the right panel that the prediction uncertainty collapses as soon as $Y$ is evaluated at two distinct points. With this behaviour, the model reflects the unicity of the solution to such ODE under two equality conditions.

\begin{figure}
\centering
\begin{subfigure}[b]{0.45\textwidth}
    \centering
    \includegraphics[width=\textwidth]{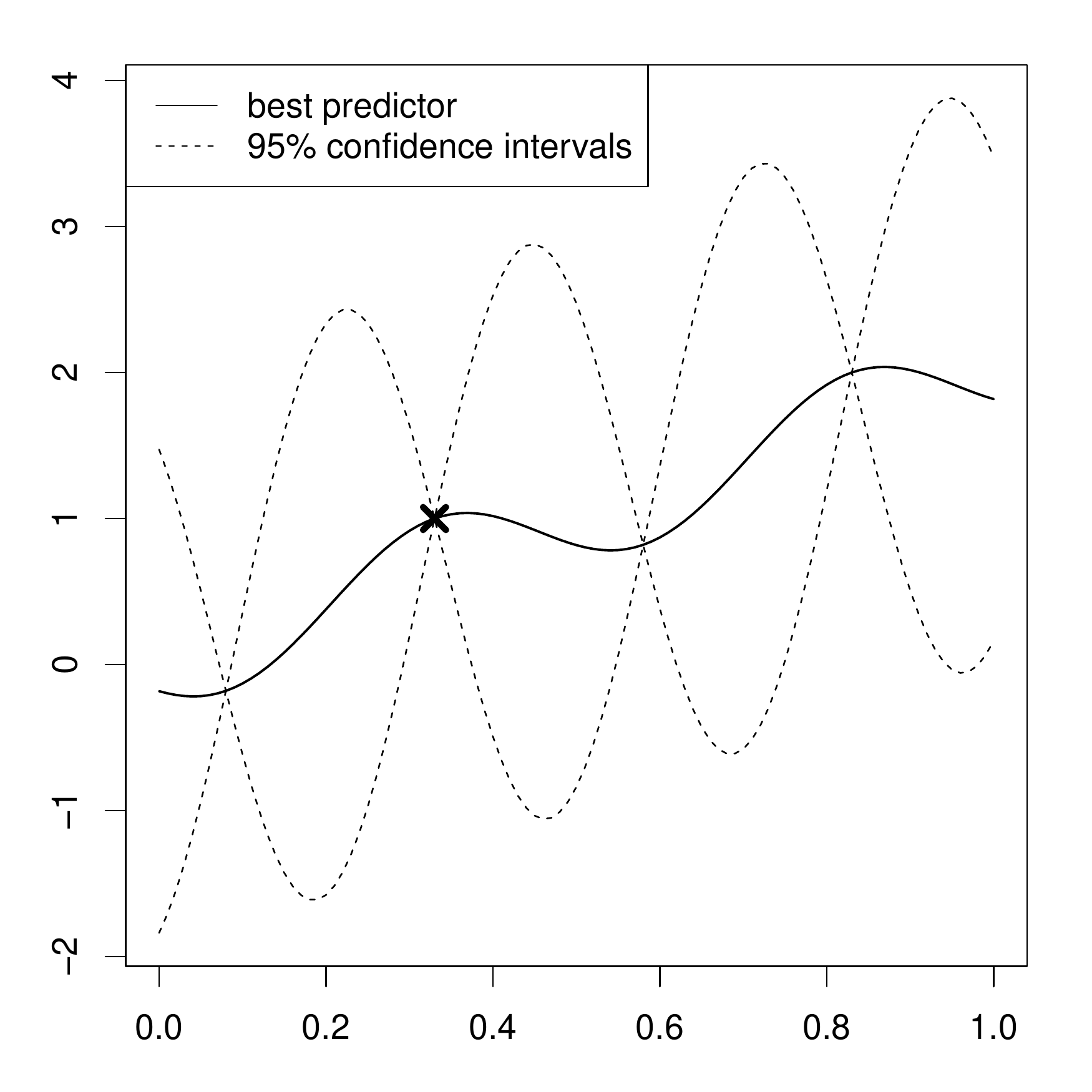}
    \caption{GPR with one data-point}
\end{subfigure} \qquad
\begin{subfigure}[b]{0.45\textwidth}
    \centering
    \includegraphics[width=\textwidth]{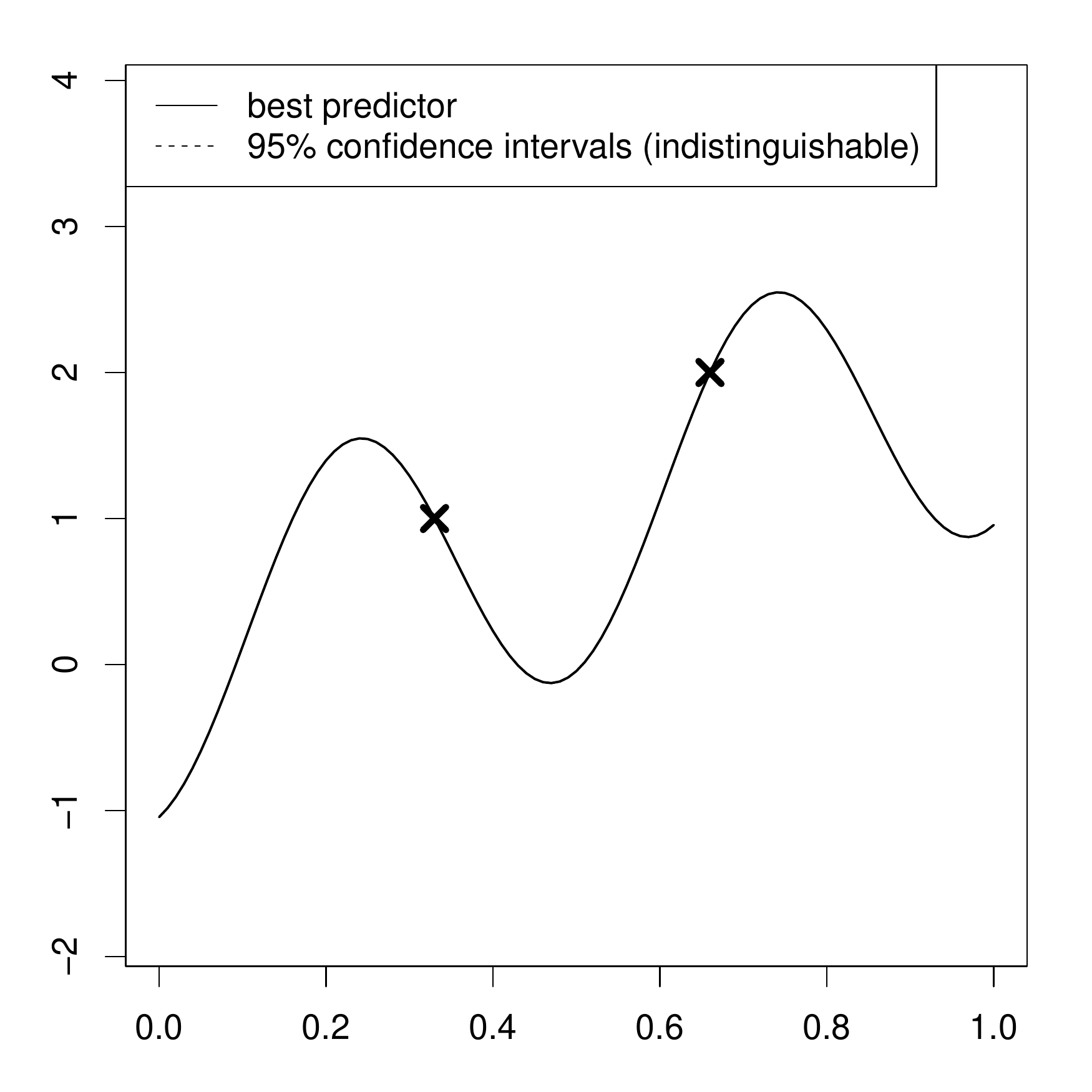}
    \caption{GPR with two data-points}
\end{subfigure}%
\caption{Examples of Gaussian process models based on a GP $Y$ satisfying $y''(x) + y(x) = 2x$. In this example, the matrix $\Sigma$ of Eq.~\ref{eq:ODEkern} is set to identity.}
\label{fig:ex_DEcond}
\end{figure}

\paragraph{GPR of an harmonic function}
We now consider the function $f(\mathbf{x}) = \cos(1-x_1) \exp (x_2)$, a solution to $\Delta y = 0$. 
As seen previously, the kernel given in Eq.~\ref{eq:kharm} satisfies this equation argumentwise so it allows to incorporate a structural prior of harmonicity in the GPR model. 
Figure~\ref{fig:ex_biHarm} shows the resulting predictions based on four observations, and the associated prediction error.
\begin{figure}
\centering
\begin{subfigure}[b]{0.45\textwidth}
    \centering
    \includegraphics[width=\textwidth]{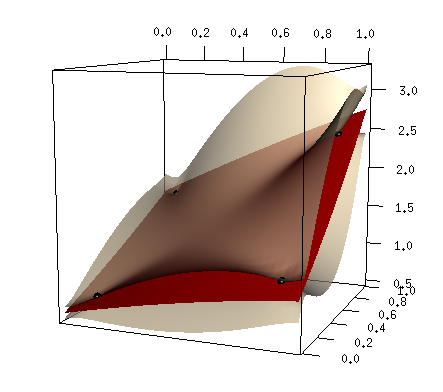}
    \caption{Mean predictor and 95\% confidence intervals}
\end{subfigure} \qquad
\begin{subfigure}[b]{0.45\textwidth}
    \centering
    \includegraphics[width=\textwidth]{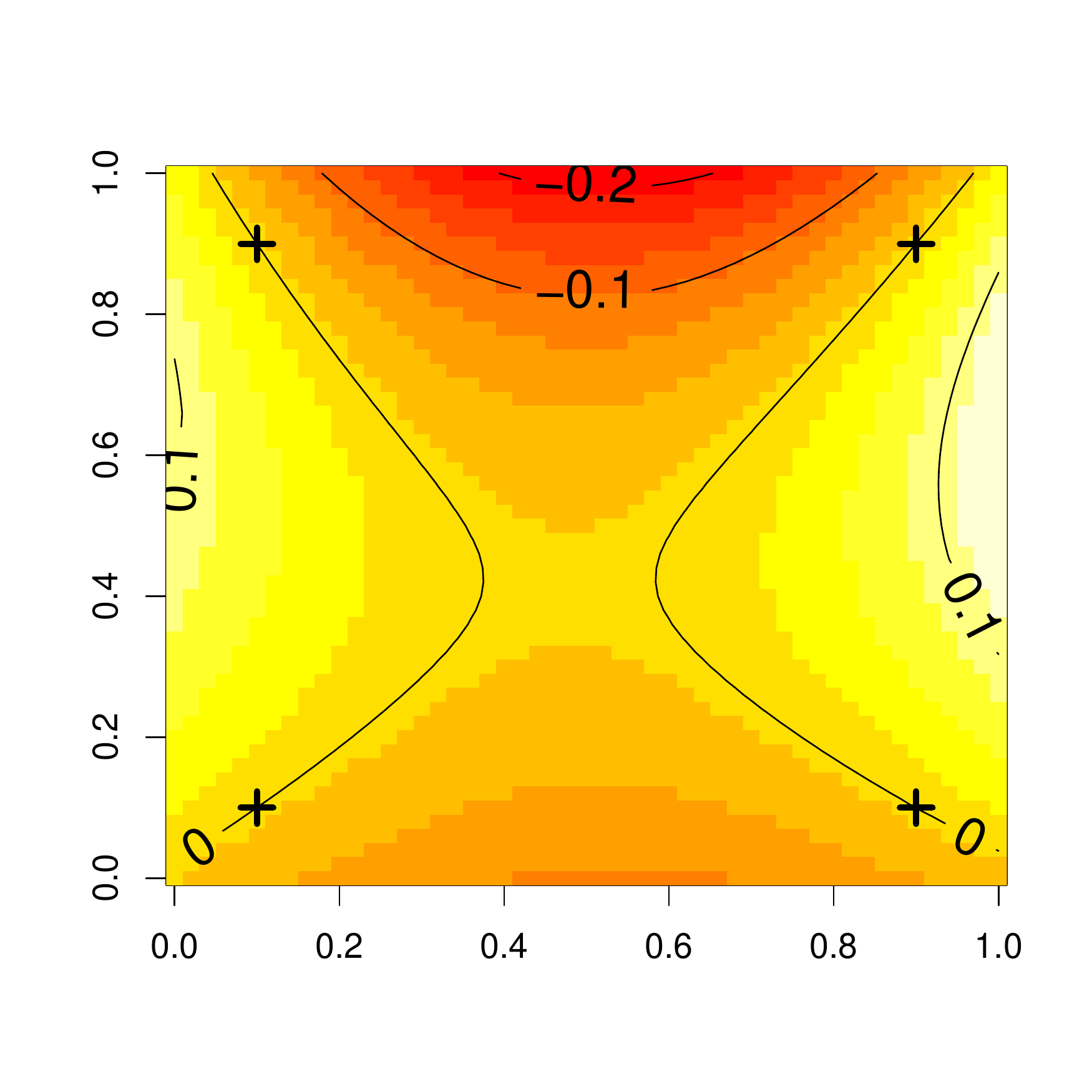}
    \caption{prediction error}
\end{subfigure}%
\caption{Example of GPR model based on an argumentwise harmonic kernel.}
\label{fig:ex_biHarm}
\end{figure}

\medskip 
Since the best predictor $m$ and $f$ are harmonic, so is also the prediction error $m - f$. It implies that the maximum error is located on the boundary of the domain (See Figure~\ref{fig:ex_biHarm}.b). This property may be of interest for the construction of design of experiments for learning harmonic functions. 

\paragraph{Sparse ANOVA kernels}
Given a product probability measure $\nu=\otimes_{1\leq i \leq d} \nu_{i}$ over $\mathbb{R}^d$, High Dimensional Model Representation 
(HDMR, See, e.g., \cite{Sobol_2003, kuo:slo:was:woz10}) corresponds to the decomposition of any $f\in L^2(\nu)$ as the sum of a constant, univariate effects, and interactions terms with increasing orders:
\begin{equation}
 	f(\mathbf{x}) = f_0 + \sum_{i=1}^d f_i(x_i) + \sum_{i<j} f_{i,j}(x_i,x_j) + \dots + f_{1,\dots,d}(\mathbf{x}) 
 	\ \ \ \ (\mathbf{x}\in D)
\label{eq:HDMR}
\end{equation}
where the $f_I$'s ($I \subset \{1,\dots,d \}$) satisfy 
$\int_{I} f(x_i) \mathrm{d}\nu_i(x_i) = 0 $ for all $i \in I$~\citep{kuo:slo:was:woz10}. This decomposition is of great interest for defining \textit{variance-based} global sensitivity indices (usually referred to as \textit{Sobol' indices}) 
quantifying the influence of each variable or group of variables on the response:
\begin{equation}
  	S_I = \frac{\mathrm{Var}[f_I(\mathbf{X}_I)]}{\mathrm{Var}[f(\mathbf{X})]}
\end{equation}
where $\mathbf{X}$ is a random vector with probability distribution $\nu$. 
Sparsity of $f$, in the sense of having many terms equal to zero in Eq.~\ref{eq:HDMR}, can be interpreted as invariance with respect to operators of the form $T(f)=f-P_I(f)$, where $P_I$ denotes the projection operator mapping $f$ to $f_I$. We will now illustrate on a popular function from the sensitivity analysis literature how taking prior knowledge of such sparsity into account may highly benefit  prediction. 

\medskip

\noindent
Sobol's $g$-function~\citep{sobol2003theorems} is defined on $[0,1]^d$ as
\begin{equation}
g(\mathbf{x}) = \prod_{i=1}^d \frac{|4x_i-2| + a_i}{1+a_i}
\end{equation}
where the $a_i$'s are arbitrary positive parameters. Beyond the convenient fact that the dimensionality $d$ is tunable, one particular feature of $g$ is that the global sensitivity indices have a closed form expression~\cite{sobol2003theorems}:
\begin{equation}
S_I = \frac{\prod_{i \in I} \beta_i}{\prod_{i=1}^d (1 + \beta_i) - 1} \text{\qquad with } \beta_i = \frac13 (1 + a_i)^{-2}.
\end{equation}
Prior knowledge of Sobol' indices allows defining a subset $\mathcal{S}$ of main effects and interactions with a significant influence on the output. 
We will consider hereafter that terms explaining less than 1e-3 \% of variance are non significant. 

\medskip

\noindent
We consider the $g$-function in ten dimensions ($d=10$) with parameter values $\mathbf{a} = (0,0,0,2,2,2,4,4,4,8)$. 
In such settings, $\mathcal{S}$ is a set of 23 subsets of indices including all the main effects except the one of $x_{10}$, and some first and second order interaction terms. 
Let us now compare prediction performances obtained with four GPR models respectively based on the following kernels:
\begin{equation}
 \begin{split}
  k_{add}(\mathbf{x},\mathbf{y}) & = \sigma_0^2 + \sum_{i = 1}^d k^0_i(x_i,y_i) \\
  k_{spa}(\mathbf{x},\mathbf{y}) & = \sigma_0^2 + \sum_{I \in \mathcal{S}}^{\phantom{d}} \prod_{i \in I} k^0_i(x_i,y_i)\\
  k_{anova}(\mathbf{x},\mathbf{y}) &= \sigma^2  \prod_{i=1}^d (k_i^0 (x_i,y_i) + 1) \\
  k_{gauss}(\mathbf{x},\mathbf{y}) & = \sigma^2 \prod_{i = 1}^d \exp \left(- \frac{(x_i- y_i)^2}{\theta_i^2} \right).
 \end{split}
\end{equation}
where the $k^0_i$ correspond to argumentwise centred Gaussian kernels as in Eq.~\ref{eq:k0}, parameterized by variances $\sigma^2_i$ and lengthscales $\theta_i$. 
The kernels $k_{add}$, $k_{spa}$, $k_{anova}$ and $k_{gauss}$ are respectively parametrized by $21$, $19$, $21$ and $11$ parameters.
Since the product in the expression of $k_{anova}$ can be expanded as a sum of $2^d$ kernels with increasing interaction orders, $k_{add}$ and $k_{spa}$ can be seen as sparse variations of $k_{anova}$ where most of the terms are set to zero. 

\medskip 

The learning set is made of $100$ uniformly distributed points over the input space. 
The parameters $\sigma^2_i$ and $\theta_i$ as well as the observation noise $\tau^2$ are estimated by maximum likelihood. 
Furthermore, a test set of 1000 uniformly distributed points is considered for assessing the model accuracies. 
The obtained results are summarized in Table~\ref{tab:last}. 
It can be seen that the model based on $k_{gauss}$ performs rather poorly on this example. This can be explained by the fact that $k_{gauss}$ does not include any bias term (i.e. a constant in the kernel expression). As a consequence, the associated model tends to come back to $0$ when the prediction point is far from the training points. 
On the other hand, models based on $k_{add}$, $k_{spa}$ and $k_{anova}$ perform significantly better since they explain at least half of the variance of $g$. 
The sum of sensitivity indices associated to the main effects shows that 66\% of the variance of $g$ is explained by its additive part so the additive structure assumed by $k_{add}$, though not completely unrealistic, is a strong assumption that disadvantages the model. 
Conversely, the structures of $k_{spa}$ and $k_{anova}$ are well-suited to the problem at hand and the associated models give the best results. This is particularly true for $k_{spa}$ which only includes the relevant terms for approximating $g$.

\begin{table}[ht!]
\begin{center}
\begin{tabular}{ccccc}
\hline
kernel & $k_{add}$ & $k_{spa}$ & $k_{anova}$ & $k_{gauss}$ \\ \hline
Log-likelihood & -32.36 & -12.65 & \textbf{-1.48} & -45.73 \\
RMSE & 0.98 & \textbf{0.74} & 0.86 & 1.17 \\ 
Q$^2$ & 0.49 & \textbf{0.71} & 0.62 & 0.28 \\ 
\hline
\end{tabular}
\caption{Comparison of the predictivity of models based on various kernels. }
\label{tab:last}
\end{center}
\end{table}

\medskip

In this example, the best model has been obtained by using a sparse kernel obtained from the knowledge of sensitivity indices. Since the latter are usually not available, the issue of automatic sparsity detection is of great importance in practice. Various methods based on a trade-off between a $L^2$-norm and a $L^1$-norm can be found in the literature (see for example~\cite{Bach2009,gunn1999supanova}).

\section{Concluding remarks and perspectives}

This article focuses on the control of pathwise invariances of square-integrable random field through their covariance structure. 
It is illustrated in various examples how a number features one may wish to impose on paths such as multivariate sparsity, symmetries, or being solution to homogeneous ODEs may be cast as invariance properties under bounded linear operators. 

\medskip
One of the main results of this work, given in Proposition~\ref{prop:kernelProcess}, relates sample path invariances to the \textit{argumentwise invariance} of the covariance kernel, in cases where $T$ is a combination of composition operators. 
Although conceptually simple, such class of operators suffices to describe various mathematical properties on functions such as invariances under finite group actions, or additivity (i.e., being sum of univariate functions). 
This result allows us in particular to extend recent results from \cite{afst_nico} by giving a complete characterization of kernels leading to centred random fields with additive paths, and also to retrieve another result from \cite{afst_david} on kernels leading to random fields with paths invariant under the action of a finite group.
Perhaps surprisingly, the obtained results linking sample paths properties to the covariance apply to squared-integrable random fields and do not restrict to the Gaussian case.

\medskip
Turning then to the particular case of Gaussian random fields, we obtain in Proposition~\ref{prop:invRKHS} a generalization of Proposition~\ref{prop:kernelProcess} to a broader class of operators,
that enables constructing Gaussian fields with paths invariant under various integral and differential operators. The core results essentially base on the Lo\`eve isometry between the Hilbert space generated by the field and its Reproducing Kernel Hilbert Space. Perspectives include revisiting those invariance results in measure-theoretic settings. 

\medskip
Taking invariances into account in random field modelling and prediction is of huge practical interest, as illustrated in Section \ref{sec:gpr}. 
Various examples involving different kinds of structural priors show how Gaussian process regression models may be drastically improved by designing an appropriate invariant kernel. 
One striking fact is that invariances assumptions may increase the accuracy of the model even if the function to approximate is not perfectly invariant. 
This can be seen on the last example where the assumed sparsity allows to improve the model by avoiding the curse of dimensionality.

\bigskip
\noindent
\textbf{Acknowledgements}: The authors would like to thank Fabrice Gamboa 
for his advice regarding the present article. 

\bigskip
\noindent
\bibliographystyle{elsarticle-num}

\end{document}